\documentclass[12pt]{amsart}
\usepackage{amsmath,amssymb,amsfonts, amscd,   
pifont, amsthm,  amsxtra, latexsym, 
eufrak}
\usepackage[all]{xy}
\usepackage{here}
\usepackage[dvips]{graphicx}

\newcommand{\RC}{\operatorname{RC}}

\theoremstyle{plain} %
  \newtheorem{theorem}{Theorem}
  \newtheorem{proposition}{Proposition}%
  \newtheorem{lemma}{Lemma} %
  \newtheorem{corollary}{Corollary} %
\theoremstyle{definition} %
  \newtheorem{example}{Example} %
  
\theoremstyle{remark} %
  \newtheorem{remark}{Remark} %
\begin{document}

\title{A short proof for Rankin--Cohen brackets and generating operators}
\author{Toshiyuki Kobayashi and Michael Pevzner } %
\maketitle %
\thispagestyle{empty}

\begin{abstract}
Motivated by the concept of 
\lq\lq{generating operators}\rq\rq\
 for a countable family of operators
 introduced in the recent paper (arXiv:2306.16800), 
 we find a method to reconstruct the Rankin--Cohen brackets from a very simple multivariable contour integral, 
 and obtain a new proof of their covariance.  
We also establish a closed formula
 of the \lq\lq{generating operator}\rq\rq\ for the Rankin--Cohen brackets
 in full generality.  
\end{abstract}



\section{Introduction}
\label{sec:intro}
Given a family of linear maps
 $R^{(\ell)} \colon \Gamma(X) \to \Gamma(Y)$ $(\ell \in {\mathbb{N}})$
 between the spaces of functions on two manifolds $X$ and $Y$, 
 the {\bf{generating operator}} $T$
 is an operator-valued
 formal power series in $t$ defined by 
\begin{equation}
\label{eqn:genop}
  T=\sum_{\ell=0}^{\infty} a_{\ell}R^{(\ell)} t^{\ell}
  \in \operatorname{Hom}_{\mathbb{C}}(\Gamma(X), \Gamma(Y)) \otimes {\mathbb{C}}[[t]], 
\end{equation}
where $a_{\ell} \in {\mathbb{C}}$ are normalizing constants.  
This concept was introduced in \cite{KPgen}
 when $a_{\ell}=\frac{1}{\ell !}$.

Special cases of the generating operators  include 
 the classical notion
 of {\bf{generating functions}}
 for orthogonal polynomials 
 which are defined in the setting where $X=\{\text{point}\}$, 
  $Y={\mathbb{C}}$, and $a_{\ell} \equiv 1$, 
 see {\it{e.g., }} \cite{AAR}, 
 whereas the {\bf{semigroups}} generated 
 by self-adjoint operators $D$ 
 as one may recall the Hille--Yosida theory, 
 correspond to the setting 
 where $X=Y$, $R^{(\ell)}=$ the $\ell$-th power of $D$, 
 and $a_{\ell}=\frac 1 {\ell !}$.  
In \cite{K23, KPgen} 
 we initiated a new line of investigation
 in the general setting 
 where $X \ne \{\text{point}\}$ and $X \ne Y$
 by taking $(X,Y)$ to be $({\mathbb{C}}^2, {\mathbb{C}})$
 as the first test case.

The idea of the generating operator is 
 to capture all the information 
 of a countable family 
 of operators
 just by a single operator.  
Its applications, 
 symbolically stated in \cite{K23} as
 \lq\lq{from discrete to continuous}\rq\rq,  
 include 

$\bullet$\enspace
 a construction 
 of non-local symmetry breaking operators
 with continuous parameters
 out of a countable family of differential operators;

$\bullet$\enspace
 a realization of an embedding 
 of discrete series representations for the two-dimensional de Sitter space
 into principal series representations of $SL(2,{\mathbb{R}})$.

In this article, 
 we find a closed formula
 for the generating operator
 of the Rankin--Cohen brackets 
 $\{\RC_{\lambda_1, \lambda_2}^{(\ell)}\}_{\ell \in {\mathbb{N}}}$
 for arbitrary $\lambda_1, \lambda_2 \in {\mathbb{N}}_+$
 by choosing appropriate normalizing constants
 $\{a_{\ell}\}_{\ell \in {\mathbb{N}}}$, 
 and the resulting formula generalizes the $\lambda_1 = \lambda_2 =1$ case
 proven in \cite{KPgen}.

We also give a new method how to find the Rankin--Cohen brackets
 by introducing their integral expression.  
The covariance property \eqref{eqn:SBO} of the Rankin--Cohen brackets
 is immediate from this viewpoint.  
The whole idea is inspired
 by the method of finding \lq\lq{generating operators}\rq\rq.

\smallskip 
{\bf{Convention.}}\enspace
${\mathbb{N}}=\{0,1,2,\cdots\}$, \,\,
${\mathbb{N}}_+=\{1,2,\cdots\}$.

\section{A short proof for Rankin--Cohen brackets}
\label{sec:RCshort}

This section introduces a complex integral transform 
 \eqref{eqn:Tl}, 
 which yields a new way to construct
 the Rankin--Cohen bidifferential operator
 $\RC_{\lambda_1, \lambda_2}^{(\ell)}$
 and a simple proof of its covariance property, 
 see Corollary \ref{cor:SBO}.

The Rankin--Cohen bracket
$
  \RC_{\lambda_1, \lambda_2}^{(\ell)} \colon {\mathcal{O}}({\mathbb{C}} \times {\mathbb{C}}) \to {\mathcal{O}}({\mathbb{C}})
$
 was originally introduced in 
\cite{xCo75, xRa56}
as a tool for constructing holomorphic modular forms of higher weights from those of lower weights. 
It is a bi-differential operator
 defined by 
\begin{equation}\label{eqn:RC}
\RC_{\lambda_1, \lambda_2}^{(\ell)}
  := \operatorname{Rest} \circ \sum_{j=0}^{\ell} (-1)^j
\begin{pmatrix} \lambda_1+ \ell-1 \\ j \end{pmatrix}
\begin{pmatrix} \lambda_2+ \ell-1 \\ \ell-j \end{pmatrix}
 \partial_1^{\ell-j} \partial_2^{j}, 
\end{equation}
where $\lambda_1, \lambda_2 \in {\mathbb{N}}_+$ and $\ell \in {\mathbb{N}}$, 
 $\partial_j=\frac{\partial}{\partial \zeta_j}$ ($j=1,2$), 
 and $\operatorname{Rest}$ denotes the restriction
 of a function in $\zeta_1$ and $\zeta_2$
 with respect to the diagonal embedding ${\mathbb{C}} \hookrightarrow {\mathbb{C}} \times {\mathbb{C}}$.

We introduce a multivariable complex integral in \eqref{eqn:Tl}, 
 with the covariance property
 (Proposition \ref{prop:23122108}), from which we obtain 
 the Rankin--Cohen brackets in Theorem \ref{thm:RCresidue}.

Let $\lambda_1, \lambda_2 \in {\mathbb{N}}_+$
 and $\ell \in {\mathbb{N}}$.  
We consider a holomorphic function 
 in $\{(\zeta_1, \zeta_2, z) \in {\mathbb{C}}^3: \zeta_1 \ne z \ne \zeta_2\}$
 defined by 
\begin{equation}
\label{eqn:Al}
   A_{\lambda_1, \lambda_2}^{(\ell)}(\zeta_1, \zeta_2;z)
   :=
   \frac{(\zeta_1-\zeta_2)^{\lambda_1 + \lambda_2 + \ell -2}}
   {(\zeta_1-z)^{\lambda_2+ \ell}(\zeta_2-z)^{\lambda_1+\ell}}.  
\end{equation}

Let $D$ be an open set in ${\mathbb{C}}$, 
 $f(\zeta_1, \zeta_2)$ a holomorphic function
 in $D \times D$, 
 and $z \in D$.  
Then the integral
\begin{equation}
\label{eqn:Tl}
   (T_{\lambda_1, \lambda_2}^{(\ell)} f)(z)
   :=
   \frac 1 {(2 \pi \sqrt{-1})^2}
   \oint_{C_1} \oint_{C_2}
   A_{\lambda_1, \lambda_2}^{(\ell)}(\zeta_1, \zeta_2;z)
   f(\zeta_1,\zeta_2) d \zeta_1 d \zeta_2
\end{equation}
is well-defined, 
 independently of the choice of contours $C_j$  ($j=1,2$)
 in $D$ around the point $z$.  
Hence one has a linear map
\[
   T_{\lambda_1, \lambda_2}^{(\ell)} \colon 
   {\mathcal{O}}(D \times D) \to {\mathcal{O}}(D).  
\]

The proof of the following integral expression of the Rankin--Cohen bracket
 actually reconstructs the explicit formula
 of the equivariant bi-differential operator 
 $\RC_{\lambda_1, \lambda_2}^{(\ell)}$
 in \eqref{eqn:RC}
 up to scalar multiplication.

\begin{theorem}
\label{thm:RCresidue}
For any $\lambda_1, \lambda_2 \ge 1$ and $\ell \in {\mathbb{N}}$,
 one has
\begin{equation*}
   T_{\lambda_1, \lambda_2}^{(\ell)}
=
 \frac {(-1)^{\lambda_1+\ell-1}(\lambda_1 + \lambda_2 + \ell -2) ! }
       {(\lambda_1+\ell-1) ! \, (\lambda_2 + \ell -1)!}
\RC_{\lambda_1, \lambda_2}^{(\ell)}.  
\end{equation*}
\end{theorem}

\begin{proof}
[Proof of Theorem \ref{thm:RCresidue}]
We iterate the residue computation for the variables $\zeta_1$ and $\zeta_2$.  
We begin with the integration
 over the first variable $\zeta_1 \in C_1$.  
\begin{align*}
  &\frac {1}{2 \pi \sqrt{-1}}
  \oint_{C_1} 
  A_{\lambda_1, \lambda_2}^{(\ell)}(\zeta_1, \zeta_2;z)f(\zeta_1, \zeta_2) d \zeta_1
\\
=& \frac{\partial_1^{\lambda_2+\ell-1}|_{\zeta_1=z}((\zeta_1-\zeta_2)^{\lambda_1+\lambda_2+\ell-2} f(\zeta_1, \zeta_2))}
  {(\lambda_2+\ell-1)!\, (\zeta_2-z)^{\lambda_1+\ell}}
\\
=& \sum_{j=0}^{\lambda_2+\ell-1}
   \frac{(-1)^{\lambda_1+\lambda_2+\ell+j} (\lambda_1+\lambda_2+\ell-2)!\, \partial_1^{\lambda_2+\ell-j-1}f(z, \zeta_2)}
        {j ! (\lambda_2 + \ell -j -1)!\, (\lambda_1+\lambda_2+\ell-j-2)!\, 
(\zeta_2-z)^{j-\lambda_2+2}}.  
\end{align*}

In turn, 
 the residue computation for the second variable $\zeta_2 \in C_2$ shows
\begin{multline*}
  \frac {1}{2 \pi \sqrt{-1}}
  \oint_{C_2} 
  \frac{\partial_1^{\lambda_2+\ell-j-1}f(z, \zeta_2)}
  {(\zeta_2-z)^{j-\lambda_2+2}}d \zeta_2
\\
=
  \begin{cases}
  \frac{1}{(j-\lambda_2+1)!}
  (\partial_1^{\lambda_2+\ell-j-1} \partial_2^{j-\lambda_2+1})f(z,z)
  \quad&\text{if $j \ge \lambda_2-1$, }
\\
  0\quad&\text{if $j < \lambda_2-1$}.  
\end{cases}
\end{multline*}

We set $r:=j-\lambda_2+1$.  
Since $\lambda_2 \ge 1$, 
 the conditions $0 \le j \le \lambda_1 + \lambda_2-2$
 and $j \ge \lambda_2-1$ are reduced to the inequality
 $0 \le r \le \ell$.  
Combining the above formulas, 
 one sees from \eqref{eqn:Tl}
 that $(-1)^{\lambda_1+\ell-1} (T_{\lambda_1,\lambda_2}^{(\ell)} f)(z)$ equals 
\begin{align*}
  &(\lambda_1+\lambda_2+\ell-2)!
   \sum_{r=0}^\ell
   \frac{(-1)^r  (\partial_1^{\ell-r} \partial_2^r f) (z,z)}
  {(\lambda_2+r-1)!\, (\ell-r)!\, (\lambda_1+\ell-r-1)!\, r!}
\\
  =&
  \frac{(\lambda_1+\lambda_2+\ell-2)!}
       {(\lambda_1+\ell-1)!\, (\lambda_2+\ell-1)!}
  \RC_{\lambda_1, \lambda_2}^{(\ell)}f(z).  
\end{align*}

Hence Theorem \ref{thm:RCresidue} is proved.  
\end{proof}

Next, 
we examine the covariance property of the kernel function 
 $A_{\lambda_1,\lambda_2}^{(\ell)}(\zeta_1, \zeta_2; z)$.  
For $g =\begin{pmatrix} a & b \\ c & d \end{pmatrix} \in SL(2,{\mathbb{C}})$, 
 we write
$
  g \cdot z :=\frac{a z + b}{c z +d}
$.  
Then one has 
\begin{equation}
\label{eqn:gz}
  g \cdot \zeta- g \cdot z
  =
  \frac{\zeta-z}{(c \zeta+d)(c z+ d)},
\end{equation}
hence the following lemma.

\begin{lemma}
\label{lem:23122107}
For $g=\begin{pmatrix} a & b \\ c & d\end{pmatrix} \in S L(2,{\mathbb{C}})$, 
 the function $A_{\lambda_1,\lambda_2}^{(\ell)}$ 
 in \eqref{eqn:Al} satisfies the following covariance property:
\[
  \frac{A_{\lambda_1,\lambda_2}^{(\ell)}(g \cdot\zeta_1, g \cdot\zeta_2;g\cdot z)}
       {A_{\lambda_1,\lambda_2}^{(\ell)}(\zeta_1, \zeta_2;z)}
  =
  \frac{(c z + d)^{\lambda_1+\lambda_2+2 \ell}}
       {(c \zeta_1 + d)^{\lambda_1-2}(c \zeta_2 + d)^{\lambda_1-2}}.  
\]
\end{lemma}

For $g \in SL(2,{\mathbb{C}})$
 such that $g \cdot D \subset D$, 
 one defines a linear map 
$\varpi_{\lambda}(g^{-1}) \colon {\mathcal{O}}(D) \to {\mathcal{O}}(g \cdot D)$
 by 
\begin{equation}
\label{eqn:pilmd}
  (\varpi_{\lambda}(g^{-1})f)(z):=(c z + d)^{-\lambda} f(\frac{az+b}{cz+d}).  
\end{equation}

Then Lemma \ref{lem:23122107} yields the following:
\begin{proposition}
\label{prop:23122108}
For any $h \in SL(2,{\mathbb{C}})$ 
 such that $D \subset h \cdot D$, 
 one has
\[
  (\varpi_{\lambda_1+\lambda_2+2\ell}(h) T_{\lambda_1,\lambda_2}^{(\ell)} f)(z)
  =T_{\lambda_1,\lambda_2}^{(\ell)}((\varpi_{\lambda_1}(h) \boxtimes \varpi_{\lambda_2}(h))f)(z).  
\]
\end{proposition}

\begin{proof}
We set $g:=h^{-1}=\begin{pmatrix} a & b \\ c & d\end{pmatrix}$.  
We note $d (g \cdot \zeta)=(c \zeta+ d)^{-2}d \zeta$.  
By definition \eqref{eqn:pilmd} and by Lemma \ref{lem:23122107}, 
$
  T_{\lambda_1,\lambda_2}^{(\ell)}((\varpi_{\lambda_1}(h) \boxtimes \varpi_{\lambda_2}(h))f)(z)
$ is equal to 
\begin{align*}
&\frac 1 {(2 \pi \sqrt{-1})^2}
  \oint_{C_1} \oint_{C_2}
  \frac{A_{\lambda_1, \lambda_2}^{(\ell)}(\zeta_1, \zeta_2; z)f(g \cdot \zeta_1, g\cdot \zeta_2)}
  {(c\zeta_1+d)^{\lambda_1} (c \zeta_2+d)^{\lambda_2}} d \zeta_1 d \zeta_2
\\
=&\frac 1 {(2 \pi \sqrt{-1})^2}
  \oint_{C_1} \oint_{C_2}
  \frac{A_{\lambda_1, \lambda_2}^{(\ell)}(g\cdot\zeta_1, g\cdot\zeta_2; g\cdot z)f(g\cdot \zeta_1, g\cdot \zeta_2)}
  {(c z+d)^{\lambda_1 +\lambda_2+2 \ell}} d (g\cdot\zeta_1) d (g\cdot\zeta_2)
\\
=&
\frac{(c z+d)^{-\lambda_1 -\lambda_2-2 \ell}}{(2 \pi \sqrt{-1})^2}
  \oint_{g\cdot C_1} \oint_{g\cdot C_2}
  A_{\lambda_1, \lambda_2}^{(\ell)}(\xi_1, \xi_2; g\cdot z)
        f(\xi_1, \xi_2)
   d \xi_1 d \xi_2
\\
=& (\varpi_{\lambda_1+\lambda_1+2 \ell}(h)T_{\lambda_1, \lambda_2}^{(\ell)}f)(z).  
\end{align*}
Thus the proposition is proved.  
\end{proof}

Since $T_{\lambda_1, \lambda_2}^{(\ell)}$ is a non-zero multiple of the Rankin--Cohen bracket
 $\RC_{\lambda_1, \lambda_2}^{(\ell)}$
 by Theorem \ref{thm:RCresidue}, 
 Proposition \ref{prop:23122108} implies the covariance property
 of the bi-differential operator $\RC_{\lambda_1, \lambda_2}^{(\ell)}$:

\begin{corollary}
\label{cor:SBO}
For any $g \in SL(2,{\mathbb{C}})$, 
 one has
\begin{equation}
\label{eqn:SBO}
  \varpi_{\lambda_1+\lambda_2+ 2 \ell}(g)\circ \RC_{\lambda_1, \lambda_2}^{(\ell)}
=
  \RC_{\lambda_1, \lambda_2}^{(\ell)}
  \circ
  (\varpi_{\lambda_1}(g) \boxtimes \varpi_{\lambda_2}(g)), 
\end{equation}
or in other words, 
\[
  d \varpi_{\lambda_1+\lambda_2+ 2 \ell}(Z)\circ \RC_{\lambda_1, \lambda_2}^{(\ell)}
=
  \RC_{\lambda_1, \lambda_2}^{(\ell)}
  \circ
  (d \varpi_{\lambda_1}(Z) \boxtimes \operatorname{id} + \operatorname{id} \boxtimes d \varpi_{\lambda_2}(Z))
\]
for any 
$
   Z=\begin{pmatrix} p & q \\ r & -p \end{pmatrix}
\in {\mathfrak{s l}}(2, {\mathbb{C}}), 
$
 where 
\[
   d \varpi_{\lambda}(Z)=-\lambda(p-r z) -(2 p z + q - r z^2)\frac d {d z}.
\]
\end{corollary}

\begin{remark}
The complete classification of the operators satisfying
 the covariance property \eqref{eqn:SBO}
 has been recently accomplished in \cite[Thm.\ 9.1]{KP2}
 by using the F-method, 
 which is based on an \lq{algebraic Fourier transform}\rq\
 of Verma modules.  
One sees from the classification that such an operator is proportional to the Rankin--Cohen bracket
 for \lq{generic}\rq\ $\lambda_1$ and $\lambda_2$, 
however, 
 there exist some other bi-differential operators
 satisfying the same property \eqref{eqn:SBO}
 for \lq{very singular}\rq\ pairs $(\lambda_1, \lambda_2)$.  
\end{remark}

\begin{remark}
Various approaches have been known 
 for the proof of the covariance property \eqref{eqn:SBO}
 for the Rankin--Cohen brackets
 since the original proof 
 by H.\ Cohen \cite{xCo75} and N.V. Kuznecov \cite{Ku75} based on the idea of Jacobi-like forms.  

D.\ Zagier \cite{Z94} proposed an  insightful proof involving theta series.
 P.\ Olver et al. \cite{Olver,OlverSan} made an observation that the Rankin--Cohen brackets can be interpreted
 as the projectivization
 of the Transvectants (\"Uberschienbung) from the classical invariant theory
 given by iterated powers
 of Cayley's $\Omega$-process, 
 see  Gordan and Gundelfinger \cite{Gordan, Gundel}.

Other approaches include a recursion relation ({\it{e.g.,}} \cite{P12})
 to find singular vectors (highest weight vectors)
 of the tensor produce of two $\mathfrak{sl}_2$-modules, 
 and a residue formula
 of a meromorphic continuation 
 of the integral symmetry breaking operators
 ({\it{e.g.,}} \cite{K18}).

A recent approach, 
 referred to as the F-method (\cite[Sect.\ 7]{KP2}), clarifies an intrinsic reason
 why the coefficients of the Rankin--Cohen brackets coincide with those of the Jacobi polynomials, 
 see \eqref{eqn:RCJacobi} below.  
J.-L.\ Clerc has proposed yet another proof 
in \cite{C17} 
using the Bernstein--Sato identity
 for the power of the determinant function
 and the intertwining property
 of the Knapp--Stein operator.  
\end{remark}

\section{Generating operators for the Rankin--Cohen Brackets}
\label{sec:genRC}

This section provides a closed formula
 of the generating operator
 for the Rankin--Cohen brackets
 $\{\RC_{\lambda_1, \lambda_2}^{(\ell)}\}_{\ell \in {\mathbb{N}}}$.  
The main result of this section is Theorem \ref{thm:23122033}, 
 which generalizes
 \cite[Thm.\ 2.3]{KPgen}
 proven in the $\lambda_1=\lambda_2=1$ case.

For $\lambda_1, \lambda_2 \in {\mathbb{N}}_+$, 
we set
\begin{equation}
\label{eqn:23122032}
  A_{\lambda_1, \lambda_2}(\zeta_1, \zeta_2;z, t)
:=
  \frac
  {(\zeta_1-\zeta_2)^{\lambda_1+\lambda_2-2}
  (\zeta_1-z)^{1-\lambda_2}
  (z-\zeta_2)^{1-\lambda_1}}
  {(\zeta_1-z) (\zeta_2-z)+t(\zeta_1-\zeta_2)}.  
\end{equation}

\begin{lemma}
\label{lem:ATaylor}
If 
$
  |t(\zeta_1-\zeta_2)|< |\zeta_1-z| \, |\zeta_2-z|
$, 
 then 
 $A_{\lambda_1, \lambda_2}(\zeta_1, \zeta_2;z,t)$ is a holomorphic function 
 of four variables
 $\zeta_1$, $\zeta_2$, $z$ and $t$, 
 and has a convergent power series expansion:
\[
  A_{\lambda_1, \lambda_2}(\zeta_1, \zeta_2;z,t)
  =\sum_{\ell=0}^{\infty}
   (-1)^{\lambda_1+\ell-1} A_{\lambda_1, \lambda_2}^{(\ell)}(\zeta_1, \zeta_2;z)t^{\ell}.  
\]
\end{lemma}

\begin{proof}
In light of the Taylor series expansion of 

\[
  \frac 1 {(\zeta_1-z) (\zeta_2-z)+t(\zeta_1-\zeta_2)}
  =
  \sum_{\ell=0}^{\infty}
  \frac{(-1)^{\ell} (\zeta_1-\zeta_2)^{\ell} t ^{\ell}}{(\zeta_1-z)^{\ell+1}(\zeta_2-z)^{\ell+1}}, 
\]
the assertion follows from the definition \eqref{eqn:23122032}.  
\end{proof}

For a domain $D$ in ${\mathbb{C}}$, 
 we set 
\begin{equation}
\label{eqn:UD}
  U_D:=\{(z,t) \in D \times {\mathbb{C}}:
  2|t|< d(z, \partial D)\}, 
\end{equation}
 where $d(z)\equiv d(z,\partial D)$ is the distance from $z \in D$
 to the boundary $\partial D$.  
We put $d(z):=\infty$ if $D={\mathbb{C}}$.  
If $D$ is the Poincar{\'e} upper half plane
$
  \Pi:=\{\zeta \in {\mathbb{C}}: \operatorname{Im} \zeta >0\}, 
$
 then $\partial D={\mathbb{R}}$
 and $d(z, \partial D)=\operatorname{Im}z$.

\begin{example}
\label{ex:UD}
{\rm{(1)}}\enspace
$U_D={\mathbb{C}} \times {\mathbb{C}}$
 if $D={\mathbb{C}}$.  
\newline
{\rm{(2)}}\enspace
$U_D=
  \{(z, t) \in {\mathbb{C}}^2: 2|t|<\operatorname{Im} z\}
$
if $D=\Pi$.  
\end{example}

As in \eqref{eqn:Tl}, 
 the integral transform 

\begin{equation}
\label{eqn:T}
 (T_{\lambda_1, \lambda_2}f)(z, t)
:=
  \frac 1 {(2 \pi \sqrt{-1})^2}
  \oint_{C_1} \oint_{C_2}
  A_{\lambda_1, \lambda_2}(\zeta_1, \zeta_2; z, t)
  f(\zeta_1, \zeta_2) d \zeta_1 d \zeta_2
\end{equation}
 defines a linear map
\begin{equation*}
   T_{\lambda_1, \lambda_2} \colon 
  {\mathcal{O}}(D \times D) \to {\mathcal{O}}(U_D).  
\end{equation*}

The integral \eqref{eqn:T} yields a generating operator
 for the Rankin--Cohen brackets
 $\{\RC_{\lambda_1, \lambda_2}^{(\ell)}\}_{\ell \in {\mathbb{N}}}$
 as below.  

\begin{theorem}
\label{thm:23122033}
The integral operator $T_{\lambda_1, \lambda_2}$ in \eqref{eqn:T}
 is expressed as 
\begin{equation}
\label{eqn:TRC}
   (T_{\lambda_1, \lambda_2}f)(z, t)
=
  \sum_{\ell=0}^{\infty}
  \frac {(\lambda_1 + \lambda_2 + \ell -2) ! \, t^{\ell}}
        {(\lambda_1+\ell-1) ! \, (\lambda_2 + \ell -1)!}
(\RC_{\lambda_1, \lambda_2}^{(\ell)} f)(z).  
\end{equation}
\end{theorem}

The following corollary is derived from Theorem \ref{thm:23122033}
 as in \cite{KPgen}.  

\begin{corollary}
$T_{\lambda_1, \lambda_2} \colon {\mathcal{O}}(D \times D)
 \to {\mathcal{O}}(D)$
 is injective for any positive integers $\lambda_1$, $\lambda_2$.  
\end{corollary}

\begin{proof}
[Proof of Theorem \ref{thm:23122033}]
Accordingly to Lemma \ref{lem:ATaylor}, 
 we expand $T_{\lambda_1, \lambda_2}f(z,t)$ into the Taylor series of $t$:
\[
   T_{\lambda_1, \lambda_2}f(z,t)
  =\sum_{\ell=0}^{\infty} (-1)^{\lambda_1+\ell-1} t^{\ell} (T_{\lambda_1, \lambda_2}^{(\ell)}f)(z)
\]
with coefficients $T_{\lambda_1, \lambda_2}^{(\ell)}f(z) \in {\mathcal{O}}(D)$.
Now the assertion follows from Theorem \ref{thm:RCresidue}.  
\end{proof}

\begin{example}
The formula \eqref{eqn:TRC} generalizes \cite[Thm.\ 2.3]{KPgen}
 which treated the case $\lambda_1=\lambda_2=1$.  
In this case, 
 $T_{1,1}$ is an integral operator 
 against the kernel
\[
   A_{1, 1}(\zeta_1, \zeta_2;z, t)
=
  \frac
  {1}
  {(\zeta_1-z) (\zeta_2-z)+t(\zeta_1-\zeta_2)}, 
\]
 and Theorem \ref{thm:23122033} reduces to 
\[
  (T_{1,1}f)(z,t)
  =\sum_{\ell=0}^{\infty} \frac{t^{\ell}}{\ell !}(R_{1,1}^{(\ell)} f)(z).  
\]
\end{example}

\begin{remark}
\label{rem:231226}
We remind a remarkable relationship 
 between the Rankin--Cohen brackets 
 and the Jacobi polynomials.  
The classical Jacobi polynomial $P_{\ell}^{(\alpha, \beta)}(x)$ is a polynomial 
 of degree $\ell$ given by 
\[
   P_{\ell}^{(\alpha, \beta)}(x)
   =
  \sum_{j=0}^{\ell} \frac{(\alpha+j+1)_{\ell-j}(\alpha+\beta+\ell+1)_{j}}{j! (\ell-j)!}
\left(\frac{x-1}2\right)^j. 
\]
Here the Pochhammer symbol $(x)_n$ is defined
 as the rising factorial $x(x+1) \cdots (x+n-1)$.  
We inflate the Jacobi polynomial 
 into a homogeneous polynomial in two variables $x$ and $y$
 of degree $\ell$ by 
\[
  \widetilde P_{\ell}^{(\alpha, \beta)}(x,y)
  :=
  y^{\ell} P_{\ell}^{(\alpha, \beta)}(1+\frac{2x}y).
\]  
Then the F-method \cite[Lem.\ 9.4]{KP2}
 establishes the correspondence:
\begin{equation}
\label{eqn:RCJacobi}
  \RC_{\lambda_1, \lambda_2}^{(\ell)}
  =
  \operatorname{Rest} \circ \widetilde P_{\ell}^{(\lambda_1-1, 1-\lambda_1-\lambda_2-2\ell)}(\frac{\partial}{\partial \zeta_1}, \frac{\partial}{\partial \zeta_2})
\end{equation}
 by showing that the \lq{symbol}\rq\ of $\RC_{\lambda_1, \lambda_2}^{(\ell)}$ satisfies
 the Jacobi differential equation
\[
  ((1-x^2)\frac{d^2}{d x^2} + (\beta-\alpha - (\alpha+\beta+2)x) \frac{d} {d x} + \ell(\ell+\alpha+\beta+1))f(x)=0
\]
where $\alpha=\lambda_1-1$ and $\beta=1-\lambda_1-\lambda_2-2\ell$.

On the other hand, 
 the generating function for the Jacobi polynomials
 $P_{\ell}^{(\alpha, \beta)}(x)$ is given by
\begin{equation}
\label{eqn:genJacobi}
\sum_{\ell=0}^{\infty} P_{\ell}^{(\alpha, \beta)}(x)t^{\ell}
  =
 \frac{2^{\alpha+\beta}}{R(1-t+R)^{\alpha}(1+t+R)^{\beta}}
\end{equation}
where $R=(1-2xt + t^2)^{\frac 1 2}$, 
 see {\it{e.g.,}} \cite[Thm.\ 6.4.2]{AAR}.  
For $\alpha=\beta=0$, 
 the Jacobi polynomial reduces to the Legendre polynomial $P_{\ell}(x)$, 
 of which the generating function is 
 given by $R^{-1}=(1-2 x t + t^2)^{-\frac 1 2}$.

However, 
 the generating function \eqref{eqn:genJacobi}
 with the normalizing constants $a_{\ell} \equiv 1$, 
 see \eqref{eqn:genop}, 
 is not directly related
 to the generating operator $T_{\lambda_1, \lambda_2}$
 defined in \eqref{eqn:TRC}
 with $a_{\ell}$ decreasing rapidly as $\ell \to \infty$, 
 see \eqref{eqn:24012803}.  
\end{remark}

\section{Generating operators for symmetry breaking}
\label{sec:SL2}

This section explains
 our results from the viewpoint
 of the representation theory.

In general, 
 the generating operator $T$ is a single operator
 that should contain all the information of a countable family operators
 $R^{(\ell)}$ ($\ell \in {\mathbb{N}}$).  
In our setting, 
 the family $\{\RC_{\lambda_1, \lambda_2}^{(\ell)}\}_{\ell \in {\mathbb{N}}}$
 of the Rankin--Cohen brackets 
 arises as {\it{symmetry breaking operators}}
 of the fusion rule of two irreducible unitary representations
 of $SL(2,{\mathbb{R}})$.  
We formulate this property
 in terms of a generating operator
 with appropriate normalizing constants $\{a_{\ell}\}_{\ell \in {\mathbb{N}}}$
 in \eqref{eqn:genop}.  
The main result of this section
is Theorem \ref{thm:24012807}.

To be precise, 
 we define the following Hilbert spaces:
\begin{alignat*}{2}
{\mathcal{H}}^2(\Pi):=&\{F \in {\mathcal{O}}(\Pi):\|F\|_{\text{Hardy}}<\infty\}
\quad
&&\text{(Hardy space), }
\\
{\mathcal{H}}^2(\Pi)_{\lambda}
:=&\{F \in {\mathcal{O}}(\Pi):\|F\|_{\lambda}<\infty\}
\quad
&&\text{(weighted Bergman spaces),   }
\end{alignat*}
where the norms are given by
\begin{align*}
\|F\|_{\text{Hardy}}^2:=&
\sup_{y>0} \int_{-\infty}^{\infty} |F(x+\sqrt{-1}y)|^2d x, 
\\
\|F\|_{\lambda}^2:=&
\int_{-\infty}^{\infty}\int_{-\infty}^{\infty} |F(x+\sqrt{-1}y)|^2 y^{\lambda-2} d x d y.  
\end{align*}
Then ${\mathcal{H}}^2(\Pi)$ and ${\mathcal{H}}^2(\Pi)_{\lambda}$ are
 invariant subspaces 
 of the representations $(\varpi_{\lambda}, {\mathcal{O}}(\Pi))$
 of $SL(2,{\mathbb{R}})$, 
 see \eqref{eqn:pilmd}, 
 for $\lambda=1$ and $\lambda \ge 2$, 
 respectively, 
 yielding irreducible unitary representations.  
By an abuse of notation, 
 we shall use the same letter $\varpi_{\lambda}$
 to denote these unitary representations.  
Then the fusion rule (abstract irreducible decomposition)
 of the tensor product representation $\varpi_{\lambda_1} \widehat \otimes \varpi_{\lambda_2}$ is known
 by Repka \cite{Re79} as follows.  
\begin{equation}
\label{eqn:fusion}
 \varpi_{\lambda_1} \widehat \otimes \varpi_{\lambda_2}
  \simeq 
  {\sum_{\ell=0}^{\infty}}{\raisebox{.8em}{\text{\footnotesize{$\oplus$}}}} \varpi_{\lambda_1+\lambda_2+2\ell}
\quad
\text{(Hilbert direct sum)},
\end{equation}
where $\widehat \otimes$ and $\sum{\raisebox{.8em}{\text{\footnotesize{$\oplus$}}}}$ stand for the Hilbert space completion
 of an algebraic tensor product
 and that of an algebraic direct sum.  
A remarkable feature in the fusion rule \eqref{eqn:fusion} is 
 that it has no continuous spectrum, 
 see \cite{xkAnn98}
 for the general theory
 of discrete decomposability
 and \cite{xrims40}
 for that of multiplicity-free decompositions.

The Rankin--Cohen bracket 
 $\RC_{\lambda_1, \lambda_2}^{(\ell)} \colon {\mathcal{O}}(\Pi \times \Pi) \to 
{\mathcal{O}}(\Pi)$
 induces a projection map
\[
   {\mathcal{H}}^2(\Pi)_{\lambda_1} \otimes {\mathcal{H}}^2(\Pi)_{\lambda_2}
   \to 
   {\mathcal{H}}^2(\Pi)_{\lambda_1 + \lambda_2+ 2 \ell}
\]
 for all $\lambda_1, \lambda_2 \in {\mathbb{N}}_+$
 and $\ell \in {\mathbb{N}}$.

We now collect these data for $\{\RC_{\lambda_1, \lambda_2}^{(\ell)}\}_{\ell \in {\mathbb{N}}}$
 in a single operator, 
 namely, 
 the generating operator.  
Let ${\underset{\ell=0}{\overset{\infty}\sum}}{\raisebox{.8em}{\text{\footnotesize{$\oplus$}}}} 
{\mathcal{H}}^2(\Pi)_{\lambda_1+\lambda_2+2\ell} \otimes {\mathbb{C}} t^{\ell}$
 denote the Hilbert completion 
 of the algebraic direct sum
\[
   {\bigoplus_{\ell=0}^{\infty}}
   {\mathcal{H}}^2(\Pi)_{\lambda_1+\lambda_2+2\ell}
   \otimes {\mathbb{C}} t^{\ell}
\]
equipped with the pre-Hilbert structure given by
\[
   (u \otimes t^{\ell}, v \otimes t^{\ell'})
   :=
   \delta_{\ell \ell'} (u,v)_{\lambda_1+\lambda_2+2\ell}.  
\]
For $\lambda_1, \lambda_2>1$, 
 we set 
\begin{equation}
\label{eqn:240128}
  a_{\ell}(\lambda_1, \lambda_2):=
 (c_{\ell}(\lambda_1, \lambda_2) r_{\ell}(\lambda_1, \lambda_2))^{-\frac 1 2}, 
\end{equation}
where we follow \cite[(2.8) and (2.9)]{KPinv}
 for the notations of positive constants $c_{\ell}(\lambda_1, \lambda_2)$ and $r_{\ell}(\lambda_1, \lambda_2)$ as below.  
\begin{align*}
  c_{\ell}(\lambda_1, \lambda_2)
  :=&
  \frac{\Gamma(\lambda_1+\ell)\Gamma(\lambda_2+\ell)}
       {(\lambda_1+\lambda_2+2\ell-1)\Gamma(\lambda_1+\lambda_2+\ell-1)\ell!}, 
\\
 r_{\ell}(\lambda_1, \lambda_2)
 :=&
  \frac{\Gamma(\lambda_1+\lambda_2+2\ell-1)}
       {2^{2\ell+2}\pi \Gamma(\lambda_1-1)\Gamma(\lambda_2-1)}.  
\end{align*}

We also set
\begin{align}
\label{eqn:24012811}
a_{\ell}(1,1)
:=&\lim_{\lambda_1 \downarrow 1} 
  \lim_{\lambda_2 \downarrow 1}
  (\lambda_1-1)^{\frac 1 2}  (\lambda_2-1)^{\frac 1 2}
  a_{\ell}(\lambda_1, \lambda_2)
\\
\notag
=&
\left(\frac{\ell !(2\ell-1)!!}
       {2^{\ell+2} \pi(2\ell+1)}
\right)^{-\frac 1 2}.  
\end{align}

\begin{theorem}
\label{thm:24012807}
\begin{enumerate}
\item[{\rm{(1)}}]
If $\lambda_1, \lambda_2>1$, 
 then the generating operator
\begin{equation}
\label{eqn:Tunitary}
T=\sum_{\ell=0}^{\infty}
a_{\ell}(\lambda_1, \lambda_2) \RC_{\lambda_1, \lambda_2}^{(\ell)}
 t^{\ell}
\end{equation}
is a unitary map
 that yields the decomposition 
\begin{equation}
\label{eqn:Hfusion}
 {\mathcal{H}}^2(\Pi)_{\lambda_1} \widehat\otimes {\mathcal{H}}^2(\Pi)_{\lambda_2}
  \overset \sim \longrightarrow 
  {\sum_{\ell=0}^{\infty}}{\raisebox{.8em}{\text{\footnotesize{$\oplus$}}}}
  {\mathcal{H}}^2(\Pi)_{\lambda_1+\lambda_2+2\ell}
  \otimes {\mathbb{C}} t^{\ell}.  
\end{equation}
\item[{\rm{(2)}}]
Similarly, 
 the generating operator \eqref{eqn:Tunitary}
 with $\lambda_1=\lambda_2=1$
 gives a unitary map
\[
 {\mathcal{H}}^2(\Pi) \widehat\otimes {\mathcal{H}}^2(\Pi)
  \overset \sim \longrightarrow
  {\sum_{\ell=0}^{\infty}}{\raisebox{.8em}{\text{\footnotesize{$\oplus$}}}}
  {\mathcal{H}}^2(\Pi)_{2\ell+2}
  \otimes {\mathbb{C}} t^{\ell}.  
\]
\end{enumerate}
\end{theorem}

\begin{remark}
The normalizing constants
 $\{a_{\ell}(\lambda_1, \lambda_2)\}_{\ell \in {\mathbb{N}}}$
 defined in \eqref{eqn:240128} are 
 different from those in \eqref{eqn:TRC}.  
However, 
 they have the same asymptotic behavior 
 as $\ell$ tends to infinity, 
 that is, 
\begin{equation}
\label{eqn:alinfty}
\underset{\ell \to \infty}\lim(a_{\ell}(\lambda_1, \lambda_2)\ell!)^{\frac 1 \ell}=1.  
\end{equation}
As we will see in Theorem \ref{thm:T} below, 
the formal power series 
\eqref{eqn:Tunitary} converges
  owing to \eqref{eqn:alinfty}.  
\end{remark}

\begin{proof}
Theorem \ref{thm:24012807} is derived from the formula of the operator norm
 of the Rankin--Cohen brackets 
 proven in \cite[Thm.\ 2.7]{KPinv}
 for $\lambda_1, \lambda_2 >1$
 and in \cite[Thm.\ 5.1]{KPgen} for $\lambda_1=\lambda_2=1$.  
\end{proof}

\section{Freedom of normalizing constants}

The definition of the generating operator in \eqref{eqn:genop}
 allows us the freedom to choose normalizing constants $\{a_{\ell}\}_{\ell \in {\mathbb{N}}}$.  
In fact, 
 the closed formula in Theorem \ref{thm:23122033}
 is obtained by taking $a_{\ell}$ to be 
 $\frac{(\lambda_1+ \lambda_2 + \ell-2)!}{(\lambda_1+ \ell-1)!\,(\lambda_2 + \ell-1)!}$
 rather than $a_{\ell}=\frac 1 {\ell !}$ or $a_{\ell} \equiv 1$.  
This section explores how the choice of $\{a_{\ell}\}_{\ell \in {\mathbb{N}}}$ 
 affects the generating operator 
 in terms of its kernel function
 by \eqref{eqn:23122021} and \eqref{eqn:hs}.

Let $h(s)$ be a holomorphic function of one variable $s$ near the origin, 
 and set
\[
     \varphi(\zeta_1, \zeta_2; z)
     :=
     \frac{\zeta_1-\zeta_2}{(\zeta_1-z) (\zeta_2-z)},   
\]

\begin{equation}
\label{eqn:23122021}
 A^{(h)}(\zeta_1, \zeta_2;z,t)
:=
  \frac{(\zeta_1-\zeta_2)^{\lambda_1+\lambda_2-2}}{(\zeta_1-z)^{\lambda_2} (\zeta_2-z)^{\lambda_1}}
  h(t \varphi(\zeta_1, \zeta_2;z)).  
\end{equation}

If $h(s)=(-1)^{\lambda_1-1} (1+s)^{-1}$, 
 then 
$
  A^{(h)}(\zeta_1, \zeta_2;z,t)
$
 in \eqref{eqn:23122021}
 coincides with $A_{\lambda_1, \lambda_2}(\zeta_1, \zeta_2;z,t)$, 
 see \eqref{eqn:23122032}.

In the generality of \eqref{eqn:23122021}, 
 an analogous covariance property
 to Lemma \ref{lem:23122107} still holds:

\begin{proposition}
\label{prop:23121932}
For any holomorphic function $h(s)$
 near the origin
 and for any $g = \begin{pmatrix} a & b \\ c & d\end{pmatrix} \in SL(2,{\mathbb{C}})$, 
 one has
\begin{multline*}
   A^{(h)}(g \cdot \zeta_1, g \cdot \zeta_2; g \cdot z,\frac{t}{(c z +d)^2})
\\
  =
  \frac{(c z +d)^{\lambda_1+\lambda_2}}
       {(c \zeta_1 +d)^{\lambda_1-2}(c \zeta_2 +d)^{\lambda_2-2}} 
  A^{(h)}(\zeta_1, \zeta_2;z,t), 
\end{multline*}
 whenever the formula makes sense.  
\end{proposition}

\begin{proof}
In view of the formula \eqref{eqn:gz}, 
one has
\[
  \varphi(g \cdot \zeta_1, g \cdot \zeta_2; g\cdot z)
 =
  (c z + d)^2 \varphi(\zeta_1, \zeta_2; z), 
\]
the proof goes in parallel to that of Lemma \ref{lem:23122107}.  
\end{proof}

Suppose that we are given a sequence $\{a_{\ell}\}_{\ell \in {\mathbb{N}}}$
 of complex numbers.  
In order to apply the above framework, 
 we define $h(s)$
 by $\{a_{\ell}\}_{\ell \in {\mathbb{N}}}$ as follows.  
We set for fixed $\lambda_1, \lambda_2 \in {\mathbb{N}}_+$
\begin{equation}
\label{eqn:23122502}
h_{\ell}:=\frac{(\lambda_1+\ell-1)!(\lambda_2+\ell-1)!}
               {(\lambda_1+\lambda_2+\ell-2)!} a_{\ell}
\quad
\text{for $\ell \in {\mathbb{N}}$}, 
\end{equation}
and define $h(s)$ by 
\begin{equation}
\label{eqn:hs}
h(s):=\underset{\ell=0}{\overset{\infty}\sum} h_{\ell} s^{\ell}.  
\end{equation}
The power series \eqref{eqn:hs} converges, 
 if $\underset{\ell \to \infty}{\limsup}|h_{\ell}|^{\frac 1 \ell}<\infty$, 
 or equivalently, 
 if 
\begin{equation}
\label{eqn:24012803}
 \frac 1 \rho:= \limsup_{\ell \to \infty} (|a_{\ell}| \ell!)^{\frac 1 \ell}
 <\infty.  
\end{equation}
Then $h(s)$ is a holomorphic function
 in $\{s \in {\mathbb{C}}:|s|<\rho\}$.  

The integral transform 
\begin{equation}
\label{eqn:T2}
  (T^{(h)}f):=\frac 1 {(2 \pi \sqrt{-1})^2} \oint_{C_1} \oint_{C_2}
  A^{(h)}(\zeta_1, \zeta_2;z, t) f(\zeta_1, \zeta_2) d \zeta_1 d \zeta_2
\end{equation}
is a generating operator of the Rankin--Cohen brackets
 $\{\RC_{\lambda_1, \lambda_2}^{(\ell)}\}_{\ell \in {\mathbb{N}}}$
 with normalizing constants $\{a_{\ell}\}_{\ell \in {\mathbb{N}}}$.  

\begin{theorem}
[Generating operator of the Rankin--Cohen brackets]
\label{thm:T}
Let $h(s)$ be defined by $\{a_{\ell}\}_{\ell \in {\mathbb{N}}}$
 as in \eqref{eqn:23122502} and \eqref{eqn:hs}.  
The integral operator $T^{(h)}$ in \eqref{eqn:T2} is expressed as 
\begin{equation}
\label{eqn:Thf}
  (T^{(h)} f) (z, t)
=\sum_{\ell=0}^{\infty}
 a_{\ell} (\RC_{\lambda_1, \lambda_2}^{(\ell)}f)(z) t^{\ell}.  
\end{equation}
In particular, 
 the operator-valued formal series
\[
  \sum_{\ell=0}^{\infty} a_{\ell} \RC_{\lambda_1, \lambda_2}^{(\ell)} t^{\ell}
\]
 converges for $|t|\ll 1$
 if $\{a_{\ell}\}_{\ell \in {\mathbb{N}}}$ satisfies
 \eqref{eqn:24012803}.  
\end{theorem}
Theorem \ref{thm:23122033} corresponds to the case
 $h(s)=(-1)^{\lambda_1-1} (1+s)^{-1}$.  

\begin{remark}
The radius of convergence of the power series \eqref{eqn:Thf}
 is zero if we take $a_{\ell}\equiv 1$.  
\end{remark}
\vskip 1pc
\par\noindent
{\bf{Acknowledgement.}}
\newline
The first author warmly thanks Professor Vladimir Dobrev
 for his hospitality 
during the 15th International Workshop:
Lie Theory and its Applications 
 in Physics, 
 held in Varna, Bulgaria 19--25, 
 June 2023.  
The authors were partially supported by the JSPS
 under the Grant-in-Aid for Scientific Research (A) 
 (JP23H00084)
 and by the JSPS Fellowship L23505.

\vskip15pt
\footnotesize{
\leftline
{Toshiyuki KOBAYASHI}
\leftline
{Graduate School of Mathematical Sciences, }
\leftline
{The University of Tokyo,
 3-8-1 Komaba, Meguro, 
Tokyo, 153-8914, Japan.} 

\&

\leftline
{French-Japanese Laboratory
 in Mathematics and its Interactions, }
\leftline{
FJ-LMI CNRS IRL2025, Tokyo, Japan}

\leftline{ E-mail:
\texttt{toshi@ms.u-tokyo.ac.jp}}

\vskip 1pc
\leftline
{Michael PEVZNER}
\leftline
{LMR, 
Universit{\'e} de Reims-Champagne-Ardenne, 
CNRS UMR 9008, F-51687,}
\leftline{Reims, France.} 

\&

\leftline
{French-Japanese Laboratory
 in Mathematics and its Interactions, }
\leftline{
FJ-LMI CNRS IRL2025, Tokyo, Japan}
\leftline{E-mail: \texttt{ pevzner@math.cnrs.fr}}
}


\begin{thebibliography}{KKM+}

\bibitem{AAR}
{\sc {G.\ Andrews, R.\ Askey, R.\ Roy}}, 
{\it{Special Functions}}, 
Cambridge, 1999.  

\bibitem{C17}
{\sc{ J.-L.\ Clerc}}, 
{\it{Covariant bi-differential operators on matrix space}}, 
\emph{Ann. Inst. Fourier}, (2017), {\bf{67}}, pp.1427--1455.

\bibitem{xCo75}
{\sc {H.\ Cohen}}, 
{\it{Sums involving the values at negative integers of $L$-functions 
 of quadratic characters}}, 
Math.\ Ann.\ {\bf{217}} (1975), 271--285.  

\bibitem{Gordan} 
{\sc{P.\ Gordan}}, 
{Invariantentheorie}, Teubner, Leipzig, 1887.

\bibitem{Gundel} 
{\sc{S.\ Gundelfinger}}, 
{\it{Zur der bin\"aren Formen}}, J.\ Reine Angew.\ Math., 
\textbf{100} (1886), pp. 413--424.

\bibitem{xkAnn98}
{\sc{T.~Kobayashi}}, 
{\textit{Discrete decomposability of the restriction of
             $A_{\frak q}(\lambda)$
            with respect to reductive subgroups {\rm{II}}---micro-local analysis and asymptotic $K$-support}}, 
Ann. of Math. (2), 
{\bf {147}} 
(1998), 
{709--729}.  

\bibitem{xrims40}
{\sc{T. Kobayashi}}, 
\textit{Multiplicity-free representations and visible actions
on complex manifolds}, 
Publ. Res. Inst. Math. Sci. 
{\textbf{41}} (2005),
{pp.~497--549},
special issue commemorating the fortieth anniversary of the founding of RIMS.

\bibitem{K18}
{\sc{T.\ Kobayashi}}, 
{\it{Residue formula for regular symmetry breaking operators}}, 
Contemporary Mathematics, 
{\bf{714}}, (2018), 
175--197, 
Amer.\ Math.\ Soc.

\bibitem{K23}
{\sc {T.~Kobayashi}}, 
{\it{Generating operators of symmetry breaking -- from
 discrete to continuous}},
To appear in Indag.\ Math.\
Available also at ArXiv:2307.16587.

\bibitem{KP1}
{\sc{T.\ Kobayashi, M.\ Pevzner}},
\textit{Differential symmetry breaking operators}. 
I. \textit{General theory and F-method}, 
Selecta Math.
 (N.S.) {\bf{22}} (2016), 
{801--845}.

\bibitem{KP2}
{\sc{T.\ Kobayashi, M.\ Pevzner}},
\textit{Differential symmetry breaking operators}. 
II. \textit{Rankin--Cohen operators for symmetric pairs}, 
Selecta Math.
 (N.S.) {\bf{22}} (2016), 
{847--911}.

\bibitem{KPinv}
{\sc{T.\ Kobayashi, M.\ Pevzner}},
\textit{Inversion of Rankin--Cohen operators via holographic transform},  
Ann.\ Inst.\ Fourier 
{\bf{70}} (2020), 
{2131--2190}.

\bibitem{KPgen}
{\sc{T.\ Kobayashi, M.\ Pevzner}},
\textit{A generating operator for Rankin--Cohen brackets}, 
arXiv: 2306.16800. 

\bibitem{Ku75}
{\sc {N.~V.~Kuznecov,}}
\textit{A new class of identities for the Fourier coefficients of modular forms.} (Russian), 
Acta Arith. \textbf{27} (1975), pp.505--519. 

\bibitem{Olver} 
{\sc{P.\ J.\ Olver}}, 
Classical Invariant Theory, London
Math. Society Student Texts \textbf{44}, Cambridge University
Press, 1999.

\bibitem{OlverSan} 
{\sc{P.\ J.\ Olver, J.A.~Sanders}},
{\it{Transvectants, modular forms, and the Heisenberg algebra}}, 
Adv.\ in Appl.\ Math., \textbf{25} (2000), 252--283.

\bibitem{P12} 
{\sc{M.\ Pevzner}},
\textit{Rankin--Cohen brackets and representations of conformal Lie groups}. Annales Math. B. Pascal, {\bf{19}} (2012)
455-484.

\bibitem{xRa56}
{\sc{R.\ A.\ Rankin}}, 
{\it{The construction of automorphic forms from the derivatives
 of a given form}}, 
 J.\ Indian Math.\ Soc. {\bf{20}} (1956), 103--116.  

\bibitem{Re79}
{\sc{J.\ Repka}}, 
\textit
{Tensor products of unitary representations of $\rm{SL}_2({\mathbb{R}})$}, 
Amer.\ J.\ Math. {\textbf{100}} (1978), 
747--774.


\bibitem{Z94}
   {\sc{D.~Zagier}},
     \textit{Modular forms and differential operators},
      {K. G. Ramanathan memorial issue},
    {Proc. Indian Acad. Sci. Math. Sci.},
   {Indian Academy of Sciences. Proceedings. Mathematical
              Sciences},
    \textbf{104},
      {(1994)},
     57--75.


\end{thebibliography}
\end{document}